\documentclass[a4paper,12pt]{article}

\usepackage{amsmath,mathrsfs, eufrak,amssymb,amsthm,latexsym,booktabs,enumitem,color,multirow,url,verbatim,microtype,tabularx,authblk}
\numberwithin{equation}{section}

\def\vect#1{\mbox{\boldmath $#1$}}

\def\e{\varepsilon}

\def\a{\alpha}

\def\til{~}
\def\C{\mathbb{C}}
\def\R{\mathbb{R}}
\def\N{\mathbb{N}}
\def\Z{\mathbb{Z}}

\def\H{\mathfrak{H}}
\def\V{\mathcal{V}}
\def\D{\mathcal{D}}
\def\KG{\mathcal{G}}
\newcommand{\n}[1]{\left\lVert#1\right\rVert}

\DeclareMathOperator*{\sgn}{sgn}
\DeclareMathOperator*{\dom}{dom}

\newtheorem{theorem}{Theorem}

\newtheorem{lemma}{Lemma}
\newtheorem{corollary}{Corollary}
\newtheorem*{conj}{Conjecture}

\theoremstyle{definition}
\newtheorem{remark}{Remark}[section]
\newtheorem{hypothesis}{Assumption}

\newtheorem{hypothesisx}{Assumption}

\renewcommand{\footnoterule}{%
	\kern -3pt
	\hrule width \textwidth height 1pt
	\kern 2pt
}

\usepackage[normalem]{ulem}
\definecolor{DarkGreen}{rgb}{0,0.5,0.1} 

\newcommand\soutD{\bgroup\markoverwith
{\textcolor{DarkGreen}{\rule[.5ex]{2pt}{1pt}}}\ULon}
\newcommand{\Hm}[1]{\leavevmode{\marginpar{\tiny%
$\hbox to 0mm{\hspace*{-0.5mm}$\leftarrow$\hss}%
\vcenter{\vrule depth 0.1mm height 0.1mm width \the\marginparwidth}%
\hbox to
0mm{\hss$\rightarrow$\hspace*{-0.5mm}}$\\\relax\raggedright #1}}}

\title{Localization of eigenvalues {for non-self-adjoint} \\ Dirac and Klein-Gordon operators}

\author{\mbox{$^*$P. D'Ancona}, \mbox{$^\dag$L. Fanelli}, \mbox{$^\ddag$D. Krej\v ci\v r\'ik}, and \mbox{$^*$N.M. Schiavone}}

\date{\vspace{-6ex}}

\begin{document}
\maketitle

\begin{center}
	\footnotesize
	
	\centerline{$^*$Department of Mathematics ``Guido Castelnuovo'', University of Rome ``La Sapienza'',}
	\centerline{Piazzale Aldo Moro 5, 00185 Rome, Italy}
	\centerline{$^\dag$IkerBasque \& Universidad del País Vasco/Euskal Herriko Unibertsitatea,} 
	\centerline{Barrio Sarriena s/n, 48940 Leioa, Bilbao, Spain}
	\centerline{$^\ddag$Department of Mathematics, 
	{Faculty of Nuclear Sciences and Physical Engineering,}}
	\centerline{Czech Technical University in Prague,
	Trojanova 13, 12000 Prague 2, Czechia}%
	\let\thefootnote\relax\footnote{
		\emph{E-mail addresses:} \mbox{dancona@mat.uniroma1.it (P. D'Ancona)}, 
		\mbox{luca.fanelli@ehu.es (L. Fanelli)}, 
		\mbox{david.krejcirik@fjfi.cvut.cz (D. Krej\v ci\v r\'ik)},
		\mbox{schiavone@mat.uniroma1.it (N.M. Schiavone)}
	}

	
	\quad
	
	
	\begin{tabularx}{0.7\textwidth}{lX}
		\begin{tabular}{@{}l@{}}
			{\bf Keywords:}
			\\
			\quad
		\end{tabular}
		& 
		\begin{tabular}{@{}l@{}}
			non-selfadjoint,
			localization of eigenvalues,
			Dirac operator,
			\\
			Klein-Gordon operator,
			Birman-Schwinger principle 
		\end{tabular}
		\\[0.35cm]
		{\bf MSC2020:}
		&
		primary 35P15, 35J99, 47A10, 47F05, 81Q12
	\end{tabularx}

\end{center}


\quad


\begin{abstract}
This note aims to give prominence to some new results on the absence and localization of eigenvalues for the Dirac and Klein-Gordon operators, starting from known resolvent estimates already established in the literature combined with the renowned Birman-Schwinger principle.
\end{abstract}


\section{Introduction}
{Since around the turn of the millennium, 
there has been a tremendous rise in interest 
for spectral properties of non-self-adjoint operators in quantum mechanics.
The physical relevance relies, inter alia, on the new concept 
of representing observables by operators which are merely
similar to self-adjoint ones,
while the mathematical community is challenged 
by the absence of the spectral theorem and variational tools.}

To make up for this lack {of conventional methods}, 
the key tool usually employed {in the non-selfadjoint settings}
is the  {celebrated} Birman-Schwinger principle
(see, e.g., 
\cite{Fra11,CLT14,Enblom16,FS17,Cue17,FKV18,FK19,CIKS20,DFS20,CIKS20} 
to cite just few recent  {works}). 
Roughly speaking (see below for precise statements), 
the principle states that~$z$ is an eigenvalue 
 {of an operator $H := H_0 + B^*A$ 
if and only if $-1$ is an eigenvalue of the Birman-Schwinger operator 
$K_z := A(H_0-z)^{-1}B^*$}.
 {In typical quantum-mechanical examples,
$H_0$ is a differential operator representing
the kinetic energy of the system, 
while $B^*A$ is a factorization of an operator of multiplication
representing an electromagnetic interaction. 
In this way, the spectral problem for an unbounded differential operator
is reduced to a bounded integral operator.
In particular,} the eigenvalues of the perturbed operator~{$H$}
are confined in the complex region defined by 
$1 \le \n{K_z}$ and the point spectrum is empty if $\n{K_z}<1$ 
uniformly  {with} respect to $z$.
 {We refer to~\cite{HK20} for an abstract
spectral-theoretic machinery based on the Birman-Schwinger idea.}

 {It is clear from}
the definition of the Birman-Schwinger operator
{that this approach reduces to establishing suitable}
resolvent estimates for the unperturbed operator~{$H_0$}.
 {Indeed,} once we know how to bound $(H_0-z)^{-1}$, 
it is  {usually} an easy matter setting~$A$ and~$B$ 
in a suitable normed space, 
and then obtain an estimate for~$K_z$. Of course, this na\"if reasoning is well-known, 
and can be synthesized claiming that each resolvent estimate corresponds, via the Birman-Schwinger principle,  {to} a localization estimate for the eigenvalues of the perturbed operator.

Since resolvent estimates have been an object of study for a  {considerably} longer time  {with} respect to the eigenvalues confinement for non-selfadjoint operators, it is natural that some results for the latter problem, even if interesting per se, go unnoticed. The goal of the current note is indeed bringing to light some new spectral results for the Dirac and Klein-Gordon operators, by inserting already established resolvent estimates in the main engine of the Birman-Schwinger principle.

The assumptions we will  {impose} on the potential are essentially pointwise smallness and decay near the origin and infinity. It would be desirable to explore other conditions, notably those involving $L^p$ norms of the potential, as in the case of Schr\"odinger operators in the seminal work by Frank \cite{Fra11}. However, for this purpose, one needs $L^p-L^q$ estimates for the resolvent, which are not easy to get and in some situations fail to hold. For example, Cuenin, Laptev and Tretter \cite{CLT14} prove in $1$-dimension their celebrated result about the eigenvalues enclosure of the perturbed Dirac operator in two (optimal) disks of the complex plane, provided that the $L^1$ norm of the potential is small. This is obtained by straightforward computations based on  {an} $L^1-L^\infty$ resolvent estimate. However, such $L^p-L^{p'}$ estimates does not hold in higher dimensions, as observed in \cite{Cue14}
 {(see also \cite{FK19})}.

 {Before moving to the main results of the present note,}
let us observe that despite the  {robustness}
of the Birman-Schwinger principle,
it is not the only tool for obtaining 
 {spectral enclosures for non-self-adjoint operators.
Indeed, another powerful technique which has been recently 
employed in various related problems is the method of multipliers
(see, e.g., \cite{FKV18,FKV2,Cossetti,CK,CFK20}).}

The  {present} note is organized as follows. 
In the next Section~\ref{Sec.main} 
we  {informally} introduce the operators 
and list our  {main} results.
 {The key resolvent estimates are collected in Section\til\ref{sec:estimates},
while the Birman-Schwinger principle is recalled in Section\til\ref{sec:BS},}
where we also properly define the perturbed operators.
 {Finally, the main results are established in Section\til\ref{sec:proof}.}


\section{Main theorems}\label{Sec.main}
 {In this note we are concerned with relativistic 
quantum-mechanical systems modelled by 
scalar Klein-Gordon and spinorial Dirac operators
in the whole space $\R^n$. They are} formally defined respectively as
\begin{align*}
\KG_{m,V} &= \KG_{m} + V,
\\
\D_{m,V} &= \D_m + V,
\end{align*}
where, for fixed mass $m\ge0$, the free Klein-Gordon and Dirac operators are
\begin{gather*}
	\KG_{m} = \sqrt{m^2-\Delta},
\\
	\D_m = -i\vect{\alpha}\cdot\nabla + m\alpha_0 
= -i \sum_{k=1}^n \alpha_k \partial_k + m\alpha_0.
\end{gather*}
The matrices $\a_k \in \C^{N\times N}$, $N:=2^{\lceil n/2 \rceil}$ being $\lceil\cdot\rceil$ the ceiling function, are elements of the Clifford algebra satisfying the anti-commutation relations
\begin{equation*}
	\alpha_j \alpha_k + \alpha_k \alpha_j = 2 \delta^j_k I_{N}
	\quad
	\text{for $j,k \in \{0,\dots,n\}$}
\end{equation*}
where $\delta^j_k$ is the Kronecker symbol. 
If we set for simplicity $N :=1 $ when we are dealing with the Klein-Gordon operator, we can say that both the operators $\KG_{m}$ and $\D_m$ act on $\H = L^2(\R^n;\C^N)$, have domain $H^1(\R^n;\C^N)$ and are self-adjoint with core $C_0^\infty(\R^n;\C^N)$.

Concerning both perturbed operators, the potential $V \colon \R^n \to \C^{N\times N}$ is a generic, possibly non-Hermitian, matrix-valued function (resp.~scalar valued in the case of Klein-Gordon).
With the usual abuse of notation, we denote with the same symbol $V$ the multiplication operator by the matrix $V$ in $\H$ with initial domain $\dom(V)= C_0^\infty(\R^n;\C^N)$. 

Moreover, for any matrix-valued function $M \colon \R^n \to \C^{N\times N}$ and norm $\n{\cdot} \colon \C\to\R_+$, we write $\n{M} := \n{|M|}$, where $|M(x)|$ denotes the operator norm of the matrix $M(x)$.

For simplicity, we will say that the spectrum of $\KG_{m,V}$ or $\D_{m,V}$ is \emph{stable} ({with} respect to the corresponding free operator spectrum) if 
\begin{equation}\label{spectrumKG}
	\sigma(\KG_{m,V}) = \sigma_c(\KG_{m,V}) = \sigma(\KG_{m}) =[m,+\infty)
\end{equation}
in the case of the Klein-Gordon operator, whereas 
\begin{gather}
	\label{spectrumD0}
	\sigma(\D_{0,V})=\sigma_c(\D_{0,V})=\sigma(\D_0)=\R \,,
	\\
	\label{spectrumDm}
	\sigma(\D_{m,V})=\sigma_c(\D_{m,V})=\sigma(\D_m)=(-\infty,-m] \cup [m,+\infty) \,,
\end{gather}
in the case of the massless and massive Dirac operators respectively. In any case, note that this means in particular that the point  {and residual spectra of the perturbed operator are empty.}

Finally, let us introduce the weights defined as
\begin{align}\label{def:tau}
\tau_\e(x) &:=
|x|^{\frac{1}{2}-\e} + |x|
%
\\
\label{def:w}
w_\sigma(x) &:= |x|(1+|\log|x||)^\sigma, 
\quad
\text{for $\sigma>1$.}
\end{align}
We are ready to enunciate the claimed results.

\begin{theorem}\label{thm:KleinGordon}
	Let $n\ge3$. There exist positive constants $\alpha$ and $\e$, independent of $V$, such that if
	\begin{equation*}
	\n{\tau_\e^{2} V}_{L^\infty} < \alpha
	\end{equation*}
	then the spectrum of
	$\KG_{m,V}$ is stable, viz. \eqref{spectrumKG} holds true.
\end{theorem}

\begin{theorem}\label{thm:Dirac}
	Let $n\ge3$. For $m=0$, there exists a positive constant $\alpha$, independent of $V$, such that if
	\begin{equation*}
	\n{w_{\sigma} V }_{L^\infty} < \alpha
	\end{equation*}
	then the spectrum of
	$\D_{0,V}$ is stable, viz. \eqref{spectrumD0} holds true.

	For $m>0$, there exist positive constants $\alpha$ and $\e$, independent of $V$, such that if
	\begin{equation*}
	\n{ \tau_\e^{2} V }_{L^\infty} < \alpha
	\end{equation*}
	then the spectrum of
	$\D_{m,V}$ is stable, viz. \eqref{spectrumDm} holds true.
\end{theorem}

For the Dirac operator 
we can improve the above theorem in two ways.  {Firstly,} slightly generalizing the choice of the weights (see also Remark\til\ref{rem:weights} below).  {Secondly,} and above all, we can give a quantitative form for the smallness condition of the potential
 {(even if our expression for the constant is probably far from being optimal)}.
 {With} this aim we bring into play the dyadic norms defined\til as
\begin{align}
\label{def:dyadicnorm}
\n{u}_{\ell^p L^q}
:=
\left( \sum_{j\in\Z} \n{u}^p_{L^q(2^{j-1} \le |x| < 2^j)} \right)^{1/p},
\qquad
\n{u}_{\ell^\infty L^q}
:=
\sup_{j\in\Z} \n{u}_{L^q(2^{j-1} \le |x| < 2^j)},
\end{align}
for $1 \le p<\infty$ and $1 \le q\le\infty$.

\begin{theorem}\label{thm:Dirac-weighted}
	Let $n\ge3$, $m\ge0$ and $\rho\in \ell^2 L^\infty(\R^n)$ be a positive weight. If $m>0$, assume in addition that $|x|^{1/2} \rho \in L^\infty(\R^n)$.
	For $m>0$, define
	\begin{equation*}\label{def:C1}
		\begin{split}
		C_1 \equiv C_1(n,m,\rho) :=\,&
		576n 
		\left[ \sqrt{n} +
		(2m+1) \sqrt[4]{64n+324}
		\right]
		\n{\rho}^2_{\ell^2 L^\infty}
		\\
		&+(2m+1) \sqrt{\frac{\pi}{2(n-2)}}
		\n{|x|^{1/2} \rho}_{L^\infty}^2
		\end{split} 
	\end{equation*}
	whereas if $m=0$,
	\begin{align}\label{def:C2}
	C_1 \equiv C_1(n,0,\rho) := 2C_2 \n{\rho}^2_{\ell^2 L^\infty},
	\quad
	C_2 \equiv C_2(n) := 576 n \max\{\sqrt{n}, \sqrt[4]{64n+324} \}.
	\end{align}
	Supposing
	\begin{equation*}
	C_1 \n{ |x| \rho^{-2} V }_{L^\infty} < 1
	\end{equation*}
	then the spectrum of
	$\D_{m,V}$ is stable, viz. \eqref{spectrumDm} holds true.
\end{theorem}

In the massless case, we can ask for less stringent conditions on the potential in order to still get the spectrum  {stable}.

\begin{theorem}\label{thm:Dirac-diadyc-0}
	Let $n\ge3$, $m=0$ and
	\begin{equation*}
		2 C_2 \n{|x|V}_{\ell^1 L^\infty} <1,
	\end{equation*}
	where $C_2$ is defined in \eqref{def:C2}. Then the spectrum of
	$\D_{0,V}$ is stable, viz. \eqref{spectrumD0} holds true.
\end{theorem}

Last but not least, we prove some results on the eigenvalues confinement in two complex disks for the massive Dirac operator.
To this end one can use either the weighted dyadic norm (this gives the counterpart for $m>0$ of Theorem\til\ref{thm:Dirac-diadyc-0}), or again the weighted-$L^2$ norm with weaker conditions on the weight $\rho$ (namely, removing in Theorem\til\ref{thm:Dirac-weighted} the assumption $|x|^{1/2}\rho \in L^\infty(\R^n)$ when $m>0$).

\begin{theorem}\label{thm:Dirac-disks}
	Let $n\ge3$, $m > 0$ and
	\begin{equation*}
		\text{$N_1(V) := \n{|x|V}_{\ell^1 L^\infty}$, \quad $N_2(V) := \n{\rho}_{\ell^2 L^\infty}^2 \n{|x| \rho^{-2} V}_{L^\infty}$}
	\end{equation*}
	for some positive weight $\rho\in\ell^2L^\infty(\R^n)$.
	For fixed $j\in\{1,2\}$, if we assume
	\begin{equation*}
	2 C_2 N_j(V) < 1,
	\end{equation*}
	with $C_2$ defined in \eqref{def:C2}, then
	\begin{equation*}
	\sigma_p(\D_{m,V}) 
	\subset
	\overline{B}_{r_0}(x_0^-)
	\cup
	\overline{B}_{r_0}(x_0^+)
	\end{equation*}
	where the two closed complex disks have centres $x_0^-, x_0^+$ and radius $r_0$ defined by
	\begin{equation*}
	x_0^\pm := \pm m \frac{\V_j^2+1}{\V_j^2-1},
	\quad
	r_0 := m \frac{2\V_j}{\V_j^2-1},
	\quad
	\text{with}
	\quad
	\V_j := \left[\frac{1}{C_2 N_j(V) } -1 \right]^2 >1.
	\end{equation*} 
\end{theorem}

\begin{remark}
	In the above Theorem\til\ref{thm:Dirac-disks}, the case $j=2$ is actually redundant. Indeed, one can easily observe that $N_1(V) \le N_2(V)$ simply by H\"older's inequality. Thus, if $2C_2N_2(V)<1$, it follows that $\V_2 \le \V_1$ and the disks obtained for $j=1$ are enclosed in those obtained for $j=2$. 
	However, we explicit both the case since, as observed above, Theorem\til\ref{thm:Dirac-disks} is in some sense the counterpart of Theorem\til\ref{thm:Dirac-weighted} and Theorem\til\ref{thm:Dirac-diadyc-0}.
\end{remark}
\begin{remark}
	In our results, the low dimensional cases $n=1,2$ are excluded. This restriction comes from the key resolvent estimates we are going to employ, collected in Lemma\til\ref{lem:KG}, Lemma\til\ref{lem:D} and Lemma\til\ref{lem:XYY*} and proved in \cite{DF08} and \cite{CDL16} (see Section\til\ref{sec:estimates} below).
	Indeed, regarding the last lemma, it holds for $n\ge3$ since to prove it the multiplier method is exploited, which fails in low dimensions. In the case of the first two lemmata instead, the low dimensions are excluded essentially due to the use of Kato-Yajima's estimates; but there is a deeper reason behind instead of a mere technical one. 
	
	In fact, tracing back the computations in \cite{DF08}, a key step in the proof of Lemma\til\ref{lem:KG} and Lemma\til\ref{lem:D} is equation (2.19) of \cite{DF08} concerning the Schr\"odinger resolvent, namely
	\begin{equation*}
		\n{\tau_\e^{-1} (-\Delta-z)^{-1} f}_{L^2} 
		\le
		C (1+|z|^2)^{-1/2} \n{\tau_\e f}_{L^2}
		\le
		C \n{\tau_\e f}_{L^2},
	\end{equation*}
	with some positive constant $C$ and $n\ge3$. The above inequality is obtained by fusing together results by Barcelo, Ruiz and Vega \cite{BRV97} and by Kato and Yajima \cite{KY89}, and it is without any doubt false for $n=1,2$. In fact, by contradiction, exploiting computations similar to the ones we will carry on in Section\til\ref{sec:proof}, one should be able to prove the counterpart of Theorems\til\ref{thm:KleinGordon}\til\&\til\ref{thm:Dirac} for the Schr\"odinger operator, in other words the spectrum of $-\Delta+V$ would be stable if $\n{\tau^2_\e V}_{L^\infty} <\alpha$ for some positive constants $\alpha$ and $\e$. This assertion is true for $n\ge3$, but certainly impossible for $n=1,2$, due to the well-know fact that the Schr\"odinger operator is critical if, and only if, $n=1,2$.
	
	The \emph{criticality} of an operator $H_0$ means that it is not stable against small perturbations: there exists a compactly
	supported potential $V$ such that $H_0 + \epsilon V$
	possesses a discrete eigenvalue for all small $\epsilon>0$.
	For the Schr\"odinger operator this is equivalent to the lack of Hardy's inequality.
	On the contrary, the existence of Hardy's inequality in dimension $n\ge3$ is sometimes referred to as the \emph{subcriticality} of $-\Delta$.
	
	In the light of this argument for the Schr\"odinger operator, a very interesting question, deserving to be pursued, naturally arises: one can conjecture that also the Klein-Gordon and Dirac operators are critical if and only if $n=1,2$, that is Theorems\til\ref{thm:KleinGordon}\til\&\til\ref{thm:Dirac} are false in low dimensions and their spectra are not stable if perturbed by small compactly supported potentials.
\end{remark}

\begin{remark}\label{rem:weights}
	In Theorem\til\ref{thm:KleinGordon}\til\&\til\ref{thm:Dirac} we used the explicit weights $\tau_\e$ and $w_\sigma$, while in the subsequent statements exploiting the weighted-$L^2$ norm they are replaced by the weight $|x|\rho^{-2}$ with $\rho\in\ell^2L^\infty(\R^n)$. We compare these assumptions.
	
	It easy to check that $\rho_1 := (1+|\log|x||)^{-\sigma/2}$ and $\rho_2 :=(|x|^{-\varepsilon}+|x|^\delta)^{-1}$ are weights in $\ell^2 L^\infty(\R^n)$ for any $\sigma>1$ and $\e,\delta>0$.
	Consequently we can set $|x| \rho^{-2} = w_\sigma(x)$ or $|x| \rho^{-2} = (|x|^{1/2-\varepsilon}+|x|^{1/2+\delta})^{2}$. The additional condition $|x|^{1/2} \rho \in L^\infty$ can be obtained for $\rho_2$ if we set $\delta=1/2$, and hence $\tau_\e^2 = |x| \rho^{-2}_2$. 
	In other words, $w_\sigma$ and $\tau_\e$ are the prototypes of the class of weights we used, since $|x|^{1/2} w_\sigma^{-1/2}, |x|^{1/2} \tau_\e^{-1} \in \ell^2 L^\infty(\R^n)$ and $|x| \tau_\e^{-1} \in L^\infty(\R^n)$.

	This generalization gives only a minor improvement in the type of admissible weights, however we think it is useful since it stresses the properties and limiting behaviors required of them. 
	
	Finally, we note that the extra condition $|x|^{1/2}\rho \in L^\infty(\R^n)$ affects the behavior of $\rho \in \ell^2 L^\infty(\R^n)$ only at infinity. Indeed, near the origin, say when $|x|\le1$, $$|x|^{1/2} \rho \le \n{\rho}_{L^\infty} \le \n{\rho}_{\ell^2 L^\infty},$$ so no further requirement is added on the behavior of $\rho$ near $x=0$; on the contrary $$\n{\rho}_{\ell^2 L^\infty(|x|\ge1)} \le \n{|x|^{-1/2}}_{\ell^2 L^\infty(|x|\ge1)} \n{|x|^{1/2}\rho}_{L^\infty}=\sqrt{2}\n{|x|^{1/2}\rho}_{L^\infty}$$
	when $|x|\ge1$, so $L^\infty(|x|\ge1) \subset \ell^2 L^\infty(|x|\ge1)$.
\end{remark}

\begin{remark}
	For a concrete example, let us make the constants $C_1$ and $C_2$ explicit in a special case. We set $n=3$, $m \in [0,1]$ and choose $\rho=|x|^{1/2} \tau_{1/2}^{-1} =(|x|^{-1/2}+|x|^{1/2})^{-1}$, which implies easily $\n{\rho}_{\ell^2 L^\infty} \le 2$ and $\n{|x|^{1/2}\rho}_{L^\infty} \le 1$.
	
	Therefore, it follows that $C_2\le 8.24 \cdot 10^3$, 
	$C_1 \le 1.11 \cdot 10^5$ 
	if $m>0$ and $C_1 \le 6.59 \cdot 10^4$    
	if $m=0$.
	Hence the smallness condition on the potential in Theorem\til\ref{thm:Dirac-weighted} is implied by
	\begin{equation*}
		\n{(1+|x|)^2 V}_{L^\infty} 
		< 
		\begin{cases}
		9.00 \cdot 10^{-6} &\text{if $m>0$,}
		\\
		1.51 \cdot 10^{-5} &\text{if $m=0$,}
		\end{cases}
	\end{equation*}
	and the one in Theorem\til\ref{thm:Dirac-diadyc-0} by
	$
		\n{|x|V}_{\ell^1 L^\infty} < 6.06 \cdot 10^{-5}.
	$
\end{remark}

{
Our conditions on the potential $V$ are certainly not sharp.
We conjecture that the pointwise smallness conditions 
of Theorem\til\ref{thm:Dirac} can be replaced by suitable 
integral hypotheses.
\begin{conj} 
Let $n=3$.
There exists a positive constant~$\alpha$ independent of~$V$	such that 
if $\|V\|_{L^3} < \alpha$ 
(respectively, $\|V\|_{L^3} + \|V\|_{L^{3/2}} < \alpha$),
then the spectrum of $\D_{0,V}$ 
(respectively, $\D_{m,V}$)
is stable, viz. \eqref{spectrumD0} 
(respectively, \eqref{spectrumDm})
holds true.
\end{conj}
}


\section{A bundle of resolvent estimates}\label{sec:estimates}

As anticipated in the Introduction, the main ingredients in our proofs are a collection of inequalities already published in the literature. 
The first two, recalled in the next two lemmata, come from \cite{DF08}.
\begin{lemma}
	\label{lem:KG}
	Let $n\ge3$ and $z\in\C$. There exist $\e>0$ sufficiently small and a constant $C>0$ such that
	\begin{align*}
	\n{ \tau_\e^{-1} (\sqrt{m^2-\Delta} -z)^{-1} f}_{L^2}
	&\le C
	\n{ \tau_\e f}_{L^2}
	\end{align*}
	where
	the weight $\tau_\e$ is defined in \eqref{def:tau}.
\end{lemma}
	The massless ($m=0$) case for this Klein-Gordon resolvent estimate is obtained by equation (2.39) in \cite{DF08} letting $W=0$. Instead, equation (2.43) from the same paper gives us the massive case for unitary mass $m=1$, and for all positive $m$ by a change of variables.
	
	Let us face now the Dirac operator.
\begin{lemma}
	\label{lem:D}
	Let $n\ge3$ and $z\in\C$. There exist $\e>0$ sufficiently small and a constant $C>0$ such that
	\begin{align}
	\label{res_est_D0}
	\n{ w_\sigma^{-1/2} (\D_0-zI_N)^{-1} f}_{L^2}
	&\le C
	\n{w_\sigma^{1/2} f}_{L^2} \,,
	\\
	\label{res_est_Dm}
	\n{ \tau_\e^{-1} (\D_m-zI_N)^{-1} f}_{L^2}
	&\le C
	\n{\tau_\e f}_{L^2} \,,
	\end{align}
	in the massless and massive case respectively, where the weights $\tau_\e$ and $w_\sigma$ are defined in \eqref{def:tau} and \eqref{def:w}.
\end{lemma}
	These estimates correspond to equation (2.49) and (2.52) from \cite{DF08} respectively, even if the estimate for the massless case was previously proved in \cite{DF07} by the same authors. 
	It should be noted that, in the cited paper, estimates (2.49) and (2.52) are explicated only in the $3$-dimensional case, but it can be easily seen that they hold in any dimension $n\ge3$, since their proofs mostly rely on the well-known identity $\D_m^2 = (-\Delta+m^2) I_N$.

The resolvent estimates just stated are uniform, in the sense that the constant $C$ in the estimates is independent of $z$. This will imply, as we will see, the total absence of eigenvalues under suitable smallness assumptions on the potential.

For the Dirac operator 
the above result can be improved. 
 {F}irst of all, we can give a non-sharp but explicit estimate for the constant $C$.  {M}oreover, paying with a constant dependent on $z$ (obtaining then a localization for the eigenvalues instead of their absence in the massless case) we can substitute the weighted-$L^2$ norms with dyadic ones, or relax the hypothesis on the weights in the massive case.

This step-up will be gained making use of the sharp resolvent estimate for the Schr\"odinger operator in dimension $n\ge3$ contained in Theorem\til1.1 of \cite{CDL16} (the same estimate can be obtained also e.g. from Theorem\til1.2 in \cite{Dan20}, but the latter does not provide explicit constants).
Setting $a=I_n$, $b=c=0$, $N=\nu=1$ and $C_a=C_b=C_c=C_-=C_+=0$ in the referred theorem, one immediately obtain the trio of estimates stated below.

\begin{lemma}\label{lem:XYY*}
	Let $n\ge 3$, $z\in\C\setminus[0,+\infty)$ and $R_0(z) := (-\Delta -z)^{-1}$. Then
	\begin{align*}
	\n{R_0(z) f}_{\dot{X}}^2 + \n{\nabla R_0(z) f}_{\dot{Y}}^2 
	&\le (288 n)^2 \n{f}_{\dot{Y}^*}^2 \,,
	\\
	|\Re z| \n{R_0(z) f}_{\dot{Y}}^2 
	&\le
	(576\sqrt{2} \, n^2)^2 \n{f}_{\dot{Y}^*}^2 \,,
	\\
	|\Im z| \n{R_0(z) f}_{\dot{Y}}^2 
	&\le
	(864\sqrt{2} \, n)^2 \n{f}_{\dot{Y}^*}^2 \,,
	\end{align*}
	where the $\dot{X}$ and $\dot{Y}$ norms are the Morrey-Campanato type norms defined by
	\begin{equation*}
	\n{u}_{\dot{X}}^2 := \sup_{R > 0} \frac{1}{R^2} \int_{|x|=R} |u|^2 dS,
	\qquad
	\n{u}_{\dot{Y}}^2 := \sup_{R > 0} \frac{1}{R} \int_{|x| \le R} |u|^2 dx,
	\end{equation*}
	and the $\dot{Y}^*$ norm is predual to the $\dot{Y}$ norm.
\end{lemma}

Since the Morrey-Campanato type norms above introduced are not so handy, observe that the $\dot{X}$ norm can be written as a {radial-angular} norm
\begin{equation*}
\n{u}_{\dot{X}}
=
\n{|x|^{-1} u}_{\ell^\infty L^\infty_{|x|} L^2_\theta}
:=
\sup_{j\in\Z} 
\sup_{R \in [2^{j-1} , 2^j)}
\n{|x|^{-1} u}_{L^2( |x|=R)} 
\end{equation*}
whereas the $\dot{Y}$ norm is equivalent to the weighted dyadic norm $\n{|x|^{-1/2} \,\cdot}_{\ell^\infty L^2}$, and hence by duality the $\dot{Y}^*$ norm is equivalent to $\n{|x|^{1/2} \,\cdot}_{\ell^1 L^2}$ (being $\n{\cdot}_{\ell^p L^q}$ defined in \eqref{def:dyadicnorm}). More precisely, since we want to show explicit constants, we have that 
\begin{equation*}
\begin{split}
\n{|x|^{-1/2} u}_{\ell^\infty L^2}^2
= \sup_{j\in\Z} \int_{2^{j-1}}^{2^j}
|x|^{-1} |u|^2 dx
\le 2 \, \sup_{j\in\Z} \frac{1}{2^j} \int_{|x|\le 2^{j}} |u|^2 dx
\le 2 \n{u}_{\dot{Y}}^2,
\end{split}
\end{equation*}
while from the other side, fixed $R \in [2^{j-1},2^j)$ for some $j\in\Z$, we get
\begin{equation*}
\begin{split}
\frac{1}{R} \int_{|x|\le R} |u|^2 dx 
\le
2^{1-j} \sum_{n=-\infty}^j 2^n \int_{2^{n-1}}^{2^n} |x|^{-1} |u|^2 dx
\le
4 \n{|x|^{-1/2}u}_{\ell^\infty L^2}^2.
\end{split}
\end{equation*}
Summarizing
\begin{align*}
{2}^{-1/2}
\n{|x|^{-1/2} u}_{\ell^\infty L^2}
\le
\n{u}_{\dot{Y}}
\le
2
\n{|x|^{-1/2} u}_{\ell^\infty L^2}
\\
{2}^{-1}
\n{|x|^{1/2} u}_{\ell^1 L^2}
\le
\n{u}_{\dot{Y}^*}
\le
2^{1/2}
\n{|x|^{1/2} u}_{\ell^1 L^2}.
\end{align*}

Inserting the above norm equivalence relations in Lemma\til\ref{lem:XYY*} one can straightforwardly infer the following.

\begin{corollary}\label{cor:dyadicS}
	Under the same assumptions of Lemma\til\ref{lem:XYY*}, the estimates
	\begin{align*}
	\n{|x|^{-1} R_0(z) f}_{\ell^\infty L^\infty_{|x|} L^2_\theta}
	&\le
	576 n \n{|x|^{1/2} f}_{\ell^1 L^2} \,,
	\\
	|z|^{1/2} \n{|x|^{-1/2} R_0(z) f}_{\ell^\infty L^2}
	&\le
	576 n \sqrt[4]{64n+324}
	\n{|x|^{1/2} f}_{\ell^1 L^2} \,,
	\\
	\n{ |x|^{-1/2} \nabla R_0(z) f}_{\ell^\infty L^2}
	&\le
	576 n
	\n{|x|^{1/2} f}_{\ell^1 L^2} \,,
	\end{align*} 
	hold true.
\end{corollary}

Simply applying H\"older's inequality, one can deduce also the {weighted-$L^2$ version} of Lemma\til\ref{cor:dyadicS}. Moreover, this allows us to employ the $-\Delta$-supersmoothness of $|x|^{-1}$ to obtain a homogeneous (in effect even stronger) weighted-$L^2$ estimate for the Schr\"odinger resolvent. Namely, we have the following.

\begin{corollary}
	\label{cor:weightedS}
	Under the same assumptions of Lemma\til\ref{lem:XYY*}, the following estimates hold
	\begin{align}
	\notag
	\n{|x|^{-3/2} \rho R_0(z)f}_{L^2}
	&\le 
	576 n \n{\rho}_{\ell^2 L^\infty}^2
	\n{|x|^{1/2} \rho^{-1} f}_{L^2} \,,
	\\
	\notag
	|z|^{1/2} \n{|x|^{-1/2} \rho R_0(z)f}_{L^2}
	&\le 
	576 n \sqrt[4]{64n+324} \n{\rho}_{\ell^2 L^\infty}^2
	\n{|x|^{1/2} \rho^{-1} f}_{L^2} \,,
	\\
	\label{eq:weightedS2}
	\n{|x|^{-1/2} \rho \nabla R_0(z)f}_{L^2}
	&\le 
	576 n \n{\rho}_{\ell^2 L^\infty}^2
	\n{|x|^{1/2} \rho^{-1} f}_{L^2} \,,
	\end{align}
	for any arbitrary positive weight $\rho \in \ell^2 L^\infty(\R^n)$.
	
	If in addition $|x|^{1/2} \rho \in L^\infty(\R^n)$, then
	\begin{equation*}
	\langle z \rangle^{1/2} \n{|x|^{-1/2} \rho R_0(z)f}_{L^2}
	\le 
	C_3
	\n{|x|^{1/2} \rho^{-1} f}_{L^2}
	\end{equation*} 
	where 
	\begin{equation*}
	C_3 \equiv C_3(n,\rho)
	:=
	576 n \sqrt[4]{64n+324} \n{\rho}_{\ell^2 L^\infty}^2
	+
	\sqrt{\frac{\pi}{2(n-2)}}
	\n{|x|^{1/2} \rho}_{L^\infty}^2
	\end{equation*}
	and $\langle x \rangle := \sqrt{1+x^2}$ are the Japanese brackets.
\end{corollary}
\begin{proof}
	By H\"older's inequality we easily obtain the set of inequalities
	\begin{align*}
	\n{ |x|^{1/2} u}_{\ell^1 L^2}
	&\le
	\n{\rho}_{\ell^2 L^\infty}
	\n{\rho^{-1} |x|^{1/2} u}_{L^2} \,,
	\\
	\n{|x|^{-1/2} \rho u}_{L^2}
	&\le
	\n{\rho}_{\ell^2 L^\infty}
	\n{|x|^{-1/2} u }_{\ell^\infty L^2} \,,
	\\
	\n{|x|^{-3/2} \rho u}_{L^2}
	&\le
	\n{|x|^{-1/2} \rho}_{\ell^2 L^2_{|x|} L^\infty_\theta} \n{|x|^{-1} u}_{\ell^\infty L^\infty_{|x|} L^2_\theta}
	\\
	&\le 
	\n{\rho}_{\ell^2 L^\infty} \n{|x|^{-1} u}_{\ell^\infty L^\infty_{|x|} L^2_\theta} \,,
	\end{align*}
	which inserted in Corollary\til\ref{cor:dyadicS} give us the first three weighted-$L^2$ estimates.
	
	The last one is instead obtained making use of the celebrated Kato-Yajima result in \cite{KY89}, that is
	\begin{equation*}
	\n{|x|^{-1} R_0(z)f}_{L^2} \le \sqrt{\frac{\pi}{2(n-2)}} \n{|x|f}_{L^2},
	\end{equation*}
	with the best constant furnished by Simon \cite{Sim92}, combined with the trivial bounds
	\begin{align*}
	\n{ |x| u}_{L^2}
	&\le
	\n{|x|^{1/2} \rho}_{L^\infty}
	\n{|x|^{1/2} \rho^{-1} u}_{L^2} \,,
	\\
	\n{|x|^{-1/2} \rho u}_{L^2}
	&\le
	\n{|x|^{1/2}\rho}_{L^\infty}
	\n{|x|^{-1} u }_{L^2} \,,
	\end{align*}
	given again by H\"older's inequality.
\end{proof}

We can return now to the Dirac operator. As a consequences of Corollaries\til\ref{cor:dyadicS}\til\&\til\ref{cor:weightedS} we obtain the following lemma.

\begin{lemma}
	\label{lem:dyadicD}
	Let $n\ge 3$ and $z\in\C\setminus\{\zeta\in\R \colon |\zeta|\ge m\}$. Then
	\begin{equation*}
	\n{|x|^{-1/2} (\D_m-z)^{-1} f}_{\ell^\infty L^2}
	\le C_2
	\left[ 1 + \left\lvert \frac{z+m}{z-m} \right\rvert^{\sgn\Re z/2} \right]
	\n{|x|^{1/2} f}_{\ell^1 L^2}
	\end{equation*}
	where $C_2$ is defined in \eqref{def:C2}, and in particular
	\begin{equation}\label{eq:Dres-nonhom}
	\n{|x|^{-1/2} \rho (\D_m-z)^{-1} f}_{L^2}
	\le C_2 \n{\rho}_{\ell^2 L^\infty}^2
	\left[ 1 + \left\lvert \frac{z+m}{z-m} \right\rvert^{\sgn\Re z/2} \right]
	\n{|x|^{1/2} \rho^{-1} f}_{L^2}
	\end{equation} 
	for any positive weight $\rho \in \ell^2 L^\infty(\R^n)$.
	
	If in addition $|x|^{1/2} \rho \in L^\infty(\R^n)$, then
	\begin{equation}\label{eq:Dres-hom}
	\n{|x|^{-1/2} \rho (\D_m-z)^{-1} f}_{L^2}
	\le C_1 
	\n{|x|^{1/2} \rho^{-1} f}_{L^2}
	\end{equation} 
	where $C_1$ is defined in the statement of Theorem\til\ref{thm:Dirac-weighted}.
\end{lemma}

\begin{proof}
	By Corollary\til\ref{cor:dyadicS} and the identity $$(\D_m -z)^{-1} = (\D_m+z) (-\Delta +m^2-z^2)^{-1} I_N$$ we obtain
	\begin{equation*}
	\begin{split}
	\n{ |x|^{-1/2} \rho (\D_m-z)^{-1} f}_{\ell^\infty L^2} 
	\le&\,
	\n{ |x|^{-1/2}\rho \sum_{k=1}^n \alpha_k \partial_k R_{0}(z^2-m^2) f}_{\ell^\infty L^2}
	\\
	&+\n{ |x|^{-1/2} \rho (m\a_0 + z I_{N}) R_{0}(z^2-m^2) f}_{\ell^\infty L^2}
	\\
	\le&\, 
	\sqrt{n}
	\n{|x|^{-1/2} \rho \nabla R_{0}(z^2-m^2) f}_{\ell^\infty L^2}
	\\
	&+ \max\{|z+m|, |z-m|\} \n{ |x|^{-1/2} \rho R_{0}(z^2-m^2) f}_{\ell^\infty L^2} 
	\\
	\le&\, 
	C_2
	\left[1 + \left|\frac{z+m}{z-m}\right|^{\sgn(\Re z)/2}\right]
	\n{|x|^{1/2} \rho^{-1} f}_{\ell^1 L^2} .
	\end{split}
	\end{equation*}
	Similarly we have the other two inequalities, using Corollary\til\ref{cor:weightedS}
	and the fact that
	\begin{equation*}
	\max\{|z+m|, |z-m|\} \langle z^2-m^2 \rangle^{-1/2} \le 2m+1
	\end{equation*}
	for the homogenous estimate \eqref{eq:Dres-hom}. Note also that, in the massless case, \eqref{eq:Dres-hom} is already contained in \eqref{eq:Dres-nonhom}.
	%
\end{proof}


\section{The Birman-Schwinger principle}\label{sec:BS}

In this section, we recall the technicalities of the Birman-Schwinger principle and we properly define an operator perturbed by a factorizable potential. We rely completely on the abstract analysis carried out by Hansmann and Krej\v ci\v r\'ik in \cite{HK20}, to which we refer for more results and background.
There, {in addition to} the point spectrum, appropriate version{s} of the principle are state{d} even for the residual, essential and continuous {spectra}.

Let us start recalling some spectral definitions. The \emph{spectrum} $\sigma(H)$ of a closed operator $H$ in a Hilbert space $\H$ is the set of the complex numbers $z$ for which $H-z \colon \dom(H) \to \H$ is not bijective. The \emph{resolvent set} is the complement of the spectrum, $\rho(H) := \C \setminus \sigma(H)$. The \emph{point spectrum} $\sigma_p(H)$ is the set of eigenvalues of $H$, namely the set of complex number such that $H-z$ is not injective. The \emph{continuous spectrum} $\sigma_c(H)$ is the set of elements of $\sigma(H) \setminus \sigma_p(H)$ such that the closure of the range of $H-z$ equals $\H$; if instead such closure is a proper subset of $H$, we speak of {the} \emph{residual spectrum} $\sigma_r(H)$. 

Here we collect the set of hypotheses we need.

\begin{hypothesis}\label{ass}
	Let $\H$ and $\H'$ be complex separable Hilbert spaces, $H_0$ be a self-adjoint operator in $\H$ and $|H_0| := (H_0^2)^{1/2}$ its absolute value. Also, let $A \colon \dom(A) \subseteq \H \to \H'$ and $B \colon \dom(B) \subseteq \H \to \H'$ be linear operators such that $\dom(|H_0|^{1/2}) \subseteq \dom(A) \cap \dom(B)$.
	
	We assume that for some (and hence for all) $b>0$ the operators $A(|H_0|+b)^{-1/2}$ and $B(|H_0|+b)^{-1/2}$
	are bounded and linear from $\H$ to $\H$.
\end{hypothesis}

	At this point, {defining} $G_0 := |H_0|+1$, we can consider, for any $z \in \rho(H_0)$, the \emph{Birman-Schwinger operator}
	\begin{equation}\label{def:Kz}
		K_z := [A G_0^{-1/2}] [G_0 (H_0-z)^{-1}] [B G_0^{-1/2}]^*,
	\end{equation}
	which is linear and bounded from $\H'$ to $\H'$.

The second assumption we need is stated below.

\begin{hypothesis}\label{ass2}
	There exists $z_0 \in \rho(H_0)$ such that 
	$
		-1 \not\in \sigma(K_{z_0}).
	$
\end{hypothesis}

While in general Assumption\til\ref{ass} is easy to check in the applications, Assumption\til\ref{ass2} is more tricky. Thus, we can replace it with the following one, stronger but more manageable.

\begin{hypothesisx}\label{ass2'}
	There exists $z_0 \in \rho(H_0)$ such that $\n{K_{z_0}}_{\H'\to\H'} <1$.
\end{hypothesisx}

That the latter implies Assumption\til\ref{ass2} can be easily proved by observing that the spectral radius is dominated by the operator norm, or recurring to Neumann series. Alternative conditions implying Assumption\til\ref{ass2} are collected in Lemma\til1 of \cite{HK20}, but for our purposes Assumption\til\ref{ass2'} will be enough.

Before recalling the Birman-Schwinger principle, we properly define the formal perturbed operator $H_0 + V$ with $V=B^*A$. 

\begin{theorem}
	Under Assumptions\til\ref{ass} and \ref{ass2}, there exists a unique closed extension $H_V$ of $H_0+V$ such that $\dom(H_V) \subseteq \dom(|H_0|^{1/2})$ and the following representation formula holds true:
	\begin{equation*}
		(\phi, H_V \psi)_{\H\to\H} = (G_0^{1/2} \phi, (H_0 G_0^{-1} + [BG_0^{-1/2}]^* AG_0^{-1/2}) G_0^{1/2}\psi)_{\H\to\H}
	\end{equation*}
	for $\phi\in\dom(|H_0|^{1/2})$, $\psi\in\dom(H_V)$.
\end{theorem}

This result correspond to Theorem\til5 in \cite{HK20}, where the operator $H_V$ is obtained via {the} pseudo-Friedrichs extension. However, an alternative approach to obtain a closed (and quasi-selfadjoint) extension of $H_0+B^*A$ is following Kato \cite{Kato66}. We refer to the paper of Hansmann and Krej\v ci\v r\' ik for a cost-benefit comparison of the two methods, and for a list of cases when the two extensions coincide.

Finally, we can exhibit the abstract Birman-Schwinger principle, for the proof of which see Theorem\til\mbox{6, 7, 8} and Corollary\til4 of \cite{HK20}.

\begin{theorem}\label{thm:BS}
	Under Assumption\til\ref{ass} and \ref{ass2}, we have:
	\begin{enumerate}[label=(\roman*)]
		\item if $z\in\rho(H_0)$, then $z \in \sigma_p(H_V)$ if and only if $-1 \in \sigma_p(K_z)$;
		
		\item if $z \in \sigma_c(H_0) \cap \sigma_p(H_V)$ and $H_V \psi = z \psi$ for $0\neq\psi\in\dom(H_V)$, then $A\psi\neq0$ and
		\begin{equation*}
			\lim_{\e\to0^\pm} (K_{z+i\e}A\psi,\phi)_{\H'\to\H'} = -(A\psi,\phi)_{\H'\to\H'} 
		\end{equation*}
		for all $\phi\in\H'$.
	\end{enumerate}
	In particular
	\begin{enumerate}[label=(\roman*)]
		\item if $z\in\sigma_p(H_V) \cap \rho(H_0)$, then $\n{K_z}_{\H'\to\H'} \ge 1$;
		
		\item if $z\in\sigma_p(H_V) \cap \sigma_c(H_0)$, then $\liminf_{\e\to0^\pm} \n{K_{z+i\e}}_{\H\to\H'} \ge 1.$
	\end{enumerate}
\end{theorem} 

While from the \lq\lq in particular'' part of the previous theorem one could infer a localization for the eigenvalues of $H_V$, the principle can be employed in a \lq\lq negative'' way to prove their absence when the norm of the Birman-Schwinger operator is strictly less than $1$ uniformly respect to $z\in\rho(H_0)$.
This is precisely stated in the next concluding result, corresponding to Theorem\til3 in \cite{HK20}, which is even richer: not only gives information on the absence of the eigenvalues, but also on the invariance of the spectrum of the perturbed operator.

\begin{theorem}
	\label{thm:BS2}
	Suppose Assumption\til\ref{ass} and that
	$
		\sup_{z\in\rho(H_0)} \n{K_z}_{\H'\to\H'} <1 .
	$
	Then we have:
	\begin{enumerate}[label=(\roman*)]
		\item $\sigma(H_0) = \sigma(H_V)$;
		\item $\sigma_p(H_V) \cup \sigma_r(H_V) \subseteq \sigma_p(H_0)$ and $\sigma_c(H_0) \subseteq \sigma_c(H_V)$.
	\end{enumerate}	
	In particular, if $\sigma(H_0)=\sigma_c(H_0)$, then $\sigma(H_V)=\sigma_c(H_V)=\sigma_c(H_0)$.
\end{theorem}

\subsection{A concrete case}

We now specialize the situation from the abstract to a concrete setting, typical in many common applications and relevant for this note.

Suppose that $\H=\H'=L^2(\R^n;\C^{N\times N})$, $N\in\N$, and $V$ is the multiplication operator generated in $\H$ by a matrix-valued (scalar-valued if $N=1$) function $V\colon\R^n\to\C^{N\times N}$, with initial domain $\dom(V)=C_0^\infty(\R^n;\C^{N})$.
As customary, we consider the factorization of $V$ given by the polar decomposition $V=UW$, where $W=\sqrt{V^* V}$ and the unitary matrix $U$ is a partial isometry. Therefore we may set $A=\sqrt{W}$, $B=\sqrt{W} U^*$ and consider the corresponding multiplication operators generated by $A$ and $B^*$ in $\H$  with initial domain $C^\infty_0(\R^n;\C^N)$, denoted by the same symbols. 
In the end, we can factorize the potential $V$ in two closed operators $A$ and $B^*$. Via the closed graph theorem, Assumption\til\ref{ass} is verified.

Furthermore, in general the operator $K_z$ defined in \eqref{def:Kz} is a bounded extension of the classical Birman-Schwinger operator $A(H_0-z)^{-1}B^*$ defined on $\dom(B^*)$. Since in our case the initial domain of $B^*$ is $C_0^\infty(\R^n;\C^{N})$, hence dense in $\H$, we get that $K_z$ is exactly the closure of $A(H_0-z)^{-1}B^*$. 

In conclusion, as suggested in the Introduction, everything reduces to the study of $\n{A(H_0-z)^{-1}B^*}_{\H\to\H}$: if there exists $z_0 \in \rho(H_0)$ such that this norm is strictly less than $1$, then Theorem\til\ref{thm:BS} holds; if this is true uniformly respect to $z\in\rho(H_0)$, then also Theorem\til\ref{thm:BS2} holds true.

\section{Proofs {of} the main theorems}
\label{sec:proof}

Taking into account the last subsection
 {and recalling the uniform resolvent estimates 
from Section\til\ref{sec:estimates}}, 
proving our claimed results on the Klein-Gordon and Dirac operators is now a simple matter.

For  $z\in\rho(H_0)$ and $\phi \in C_0^\infty(\R^n)$, from the resolvent estimate in Lemma\til\ref{lem:KG}, we immediately get
\begin{align*}
\n{A (\KG_{m}-z)^{-1} B^* \phi}_{L^2}
&\le 
\n{A \tau_\e}_{L^\infty}
\n{\tau_\e^{-1} (\KG_{m}-z)^{-1}B^* \phi}_{L^2}
\\&\le C
\n{A \tau_\e}_{L^\infty}
\n{\tau_\e B^* \phi}_{L^2}
\\&\le C
\n{\tau_\e^2 V }_{L^\infty}
\n{\phi}_{L^2}
\\&< \alpha C
\n{\phi}_{L^2}.
\end{align*}
If $\alpha=1/C$, then Theorem\til\ref{thm:KleinGordon} follows from Theorem\til\ref{thm:BS2}. By analogous computations one obtains Theorem\til\ref{thm:Dirac} making use of the resolvent estimates in Lemma\til\ref{lem:D}, and the other theorems concerning the Dirac operator exploiting Lemma\til\ref{lem:dyadicD}.

Let us just make explicit the computations for Theorem\til\ref{thm:Dirac-disks} with $N_1(V)=\n{|x|V}_{\ell^1 L^\infty}$. By Lemma\til\ref{lem:dyadicD} we have that
\begin{align*}
\n{A (\D_m-z)^{-1} B^* \phi}_{L^2}
&\le 
\n{A |x|^{1/2}}_{\ell^2 L^\infty}
\n{|x|^{-1/2} (\D_m-z)^{-1}B^* \phi}_{\ell^\infty L^2}
\\&\le C_2
\left[ 1 + \left\lvert \frac{z+m}{z-m} \right\rvert^{\sgn\Re z/2} \right]
\n{A |x|^{1/2}}_{\ell^2 L^\infty}
\n{|x|^{1/2} B^* \phi}_{\ell^1 L^2}
\\&\le C_2
\left[ 1 + \left\lvert \frac{z+m}{z-m} \right\rvert^{\sgn\Re z/2} \right]
\n{|x| V }_{\ell^1 L^\infty}
\n{\phi}_{L^2}
.
\end{align*}
Setting $\V_1 := [1 / [C_2 N_1(V)] - 1]^2 >1$, the condition $\n{A (\D_m-z)^{-1} B^* \phi}_{\H\to\H} \ge 1$ turns out to be equivalent to the expression
\begin{equation*}
	\left( \Re z - \sgn(\Re z) m \frac{\V_1^2+1}{\V_1^2-1} \right)^2
	+ (\Im z)^2
	\le 
	\left( m \frac{2\V_1}{\V_1^2-1}\right)^2
\end{equation*}
which define exactly the disks in the statement of the theorem. Just take any $z_0\in\rho(\D_m)$ outside these two disks to verify Assumption\til\ref{ass2'}, and finally we can prove the statement applying the \lq\lq in particular'' part of Theorem\til\ref{thm:BS}.


\section*{Acknowledgement}

The first, second and fourth authors are members of the Gruppo Nazionale per L'Analisi Matematica, la Probabilità e le loro Applicazioni (GNAMPA) of the Istituto Nazionale di Alta Matematica
(INdAM). 
 {The third author (D.K.) was supported
by the EXPRO grant No.~20-17749X of the Czech Science Foundation.}
The fourth author (N.M.S.) is partially supported by \textit{Progetti per Avvio alla Ricerca di Tipo 1 -- Sapienza Università di Roma}.


\bibliographystyle{plain}

\end{document}